\documentclass[titlepage, 11pt]{article}
\usepackage[margin=2.5cm]{geometry}
\usepackage{mathtools}

\usepackage{csquotes}
\usepackage{enumitem}
\MakeOuterQuote{"}

\usepackage[backend=biber,style=numeric, 
maxbibnames=20, 
url=false, 
isbn=false, 
doi=false]{biblatex}
\DeclareLabelalphaTemplate{
  \labelelement{
    \field[final]{shorthand}
    \field{label}
    \field[strwidth=3,strside=left,ifnames=1]{labelname}
    \field[strwidth=1,strside=left]{labelname}
  }
  \labelelement{
    \field{year}    
  }
}

\newbibmacro{string+doiurlisbn}[1]{%
  \iffieldundef{doi}{%
    \iffieldundef{url}{%
      \iffieldundef{isbn}{%
        \iffieldundef{issn}{%
          #1%
        }{%
          \href{http://books.google.com/books?vid=ISSN\thefield{issn}}{#1}%
        }%
      }{%
        \href{http://books.google.com/books?vid=ISBN\thefield{isbn}}{#1}%
      }%
    }{%
      \href{\thefield{url}}{#1}%
    }%
  }{%
    \href{http://dx.doi.org/\thefield{doi}}{#1}%
  }%
}

\DeclareFieldFormat{title}{%
  \usebibmacro{string+doiurlisbn}{\mkbibemph{#1}}}
\DeclareFieldFormat
  [article,inbook,incollection,inproceedings,patent,thesis,unpublished]
  {title}{\usebibmacro{string+doiurlisbn}{\mkbibquote{#1\isdot}}}
\DeclareFieldFormat
  [suppbook,suppcollection,suppperiodical]
  {title}{\usebibmacro{string+doiurlisbn}{#1}}

\usepackage{authblk}

\usepackage[hyphens]{xurl}
\usepackage[usenames, dvipsnames]{xcolor}
\definecolor{linkB}{RGB}{4, 106, 143}
\definecolor{defRED}{RGB}{84, 4, 18}
\usepackage[colorlinks=true, breaklinks=true, linkcolor={linkB}, urlcolor={linkB}, citecolor={linkB}]{hyperref}
\newcommand{\defin}[1]{\textcolor{defRED}{\emph{#1}}\index{#1}}
\usepackage[english]{babel}
\addto\extrasenglish{}
\addto\extrasenglish{}

\usepackage{amsmath, amssymb, amsthm}
\usepackage{thmtools}
\usepackage{thm-restate}

\declaretheorem[name=Lemma, numberwithin = section]{lemma}
\declaretheorem[name=Theorem,sibling = lemma]{theorem}
\declaretheorem[name=Proposition, sibling=lemma]{proposition}

\declaretheorem[name=Corollary, sibling=lemma]{corollary}

\declaretheorem[name=Problem, sibling = lemma]{problem}

\declaretheorem[name=Claim]{claim}

\newenvironment{subproof}[1][\proofname]{%
  \begin{proof}[#1]%
}{%
  \end{proof}%
}

\usepackage[noabbrev,capitalise,nameinlink]{cleveref}
\crefname{theorem}{Theorem}{Theorems}
\crefname{proposition}{Proposition}{Propositions}
\crefname{lemma}{Lemma}{Lemmas}
\crefname{claim}{Claim}{Claims}
\crefname{subclaim}{Sub-Claim}{Sub-Claims}
\crefname{observation}{Observation}{Observations}
\crefname{remark}{Remark}{Remarks}
\crefname{corollary}{Corollary}{Corollaries}
\crefname{definition}{Definition}{Definitions}
\crefname{conjecture}{Conjecture}{Conjectures}
\crefname{question}{Question}{Questions}
\crefformat{equation}{(#2#1#3)}
\Crefformat{equation}{Equation #2(#1)#3}

\usepackage{chngcntr}

\DeclareMathOperator*{\argmin}{arg\,min}
\newcommand{\cartprod}{\mathbin{\Box}}
\newcommand{\fracenter}[1]{\overset{\mathfrak{g}}{#1}}

\newcommand{\gcap}{\fracenter{\cap}}
\newcommand{\gcup}{\fracenter{\cup}}
\newcommand{\biggcap}{\fracenter{\bigcap}}
\newcommand{\biggcup}{\fracenter{\bigcup}}

\usepackage[classfont=sanserif, langfont=roman,funcfont=italic]{complexity}

\title{\LARGE{Graphs of bounded chordality}}
\author{Aristotelis Chaniotis
\\
\vspace{-0.4cm}
Department of Combinatorics and Optimization, \\University of Waterloo, Waterloo, Ontario, N2L 3G1, Canada
\\
\vspace{0.5cm}
Babak Miraftab
\\School of Computer Science, \\Carleton University, Ottawa, Ontario, K1S 5B6, Canada\\
\vspace{0.5cm}
Sophie Spirkl\thanks{We acknowledge the support of the Natural Sciences and Engineering Research Council of Canada (NSERC), [funding
reference number RGPIN-2020-03912]. \newline Cette recherche a \'et\'e financ\'ee par le Conseil de recherches en sciences naturelles et en g\'enie du Canada (CRSNG),
[num\'ero de r\'ef\'erence RGPIN-2020-03912]. \\ This project was funded in part by the Government of Ontario.}\\
Department of Combinatorics and Optimization, \\University of Waterloo, Waterloo, Ontario, N2L 3G1, Canada
}

\date{\today}

\begin{document}

\maketitle

\fontsize{12}{16}\selectfont

\begin{abstract}
A {\em hole} in a graph is an induced subgraph which is a cycle of length at least four. A graph is {\em chordal} if it contains no holes. Following McKee and Scheinerman (1993), we define the {\em chordality} of a graph $G$ to be the minimum number of chordal graphs on $V(G)$ such that the intersection of their edge sets is equal to $E(G)$. In this paper we study classes of graphs of bounded chordality.

In the 1970s, Buneman, Gavril, and Walter, proved independently that chordal graphs are exactly the intersection graphs of subtrees in trees. We generalize this result by proving that the graphs of chordality at most $k$ are exactly the intersection graphs of convex subgraphs of median graphs of tree-dimension $k$. 

A hereditary class of graphs $\mathcal{A}$ is {\em $\chi$-bounded} if there exists a function $ f\colon \mathbb{N}\rightarrow \mathbb{R}$ such that for every graph $G\in \mathcal{A}$, we have $\chi(G) \leq f(\omega(G))$. In 1960, Asplund and Gr\"unbaum proved that the class of all graphs of boxicity at most two is $\chi$-bounded. In his seminal paper "Problems from the world surrounding perfect graphs," Gy\'arf\'as (1985), motivated by the above result, asked whether the class of all graphs of chordality at most two, which we denote by $\mathcal{C}\gcap \mathcal{C}$, is $\chi$-bounded. We discuss a result of Felsner, Joret, Micek, Trotter and Wiechert (2017), concerning tree-decompositions of Burling graphs, which implies an answer to Gy\'arf\'as' question in the negative. We prove that two natural families of subclasses of $\mathcal{C}\gcap \mathcal{C}$ are polynomially $\chi$-bounded.

Finally, we prove that for every $k\geq 3$ the $k$-\textsc{Chordality Problem}, which asks to decide whether a graph has chordality at most $k$, is $\NP$-complete.
\end{abstract}

\section{Introduction}
\label{sec:introduction}
For basic notions and notation not defined here we refer readers to \cite{west2020combinatorial}.  In this paper we consider finite, undirected graphs with no loops or parallel edges. For a set $S$ we denote the power set of $S$ by $2^{S}$, and the set of all size-two elements of $2^{S}$ by $\binom{S}{2}$. Let $G$ be a graph. We call a subset of $V(G)$ a \defin{clique} (respectively\ a \defin{stable} set) of $G$ if it is a set of pairwise adjacent (respectively non-adjacent) vertices. A clique of size three is called a \defin{triangle}. The \defin{clique number} of $G$, denoted by $\omega(G)$, is the maximum size of a clique in $G$. For vertices $u,v\in V(G)$ a \defin{$(u,v)$-path} in $G$ is a path which has as ends the vertices $u$ and $v$. A \defin{non-edge} of $G$ is an element of the set $\binom{V(G)}{2} \setminus E(G)$. Given a graph $H$ we say that $G$ is \defin{$H$-free} (respectively \defin{contains $H$}) if it contains no (respectively contains an) induced subgraph isomorphic to $H$. For a set $X\subseteq V(G)$, we denote by $G[X]$ the subgraph of $G$ which is induced by $X$. A class of graphs is \defin{hereditary} if it is closed under isomorphism and under taking induced subgraphs.

Let $G_{1}, \ldots, G_{k}$ be graphs. Then, their \defin{intersection} (respectively \defin{union}), which we denote by $\cap_{i\in[k]}G_{i}$ (respectively $\cup_{i\in[k]}G_{i}$), is the graph $(\cap_{i\in[k]}V(G_{i}), \cap_{i\in[k]}E(G_{i}))$ (respectively $(\cup_{i\in[k]}V(G_{i}), \cup_{i\in[k]}E(G_{i}))$). Given graph classes $\mathcal{G}_{1}, \ldots , \mathcal{G}_{k}$, we denote by $\mathcal{G}_{1} \gcap \ldots \gcap \mathcal{G}_{k}$ the class $\{G: \exists G_{i} \in \mathcal{G}_{i} \text{ such that } G=G_{1}\cap \ldots \cap G_{k}\}$, which we call the \defin{graph intersection} of $\mathcal{G}_{1}, \ldots , \mathcal{G}_{k}$. The \defin{graph union} of $\mathcal{G}_{1}, \ldots , \mathcal{G}_{k}$, which we denote by $\mathcal{G}_{1} \gcup \ldots \gcup \mathcal{G}_{k}$, is the class $\{G: \exists G_{i} \in \mathcal{G}_{i} \text{ such that } G=G_{1}\cup \ldots \cup G_{k}\}$.

Given a class of graphs $\mathcal{A}$ and a graph $G$, we follow Kratochv\'il and Tuza \cite{kratochvil1994intersection}, and define the \defin{intersection dimension of $G$ with respect to $\mathcal{A}$} to be the minimum integer $k$ such that $G \in \biggcap_{i\in [k]} \mathcal{A}$ if such a $k$ exists, and $+\infty$ otherwise. We remark that the intersection dimension of graphs with respect to graph 
classes has been also studied by by Cozzens and Roberts \cite{cozzens1989dimensional} under a different name: they called a graph property $P$ \defin{dimensional} if for every graph $G$, the intersection dimension of $G$ with respect to the class $\mathcal{A}(P):=\{G:G \text{ has the property } P\}$ is finite. For a positive integer $n$, we denote by $K_{n}$ the complete graph on $n$ vertices, and by $K_{n}^{-}$ the graph we obtain from $K_{n}$ by deleting an edge. It is easy to observe that a graph property $P$ is dimensional if and only if for every for every positive integer $n$, both the graphs $K_{n}$ and $K_{n}^{-}$ have the property $P$.

A \defin{hole} in a graph $G$ is an induced cycle of length at least four.
A graph is \defin{chordal} if it contains no holes, and we denote the class 
of chordal graphs by $\mathcal{C}$. Following McKee and Scheinerman \cite{mckee1993chordality} we call
the intersection dimension of a graph $G$ with respect to $\mathcal{C}$ the \defin{chordality}
of $G$ and we denote it by $\mathsf{chor}(G)$.
Since, for every positive integer $n$, both
the graphs $K_{n}$ and $K_{n}^{-}$ are chordal, it follows that 
the chordality of every graph is finite (and upper bounded by the number of its non-edges). 
To the best of our knowledge, chordality was first studied by 
Cozzens and Roberts \cite{cozzens1989dimensional}
under the name rigid circuit dimension.

Given a finite family of nonempty sets 
$\mathcal{S}$, the \defin{intersection graph} of $\mathcal{S}$ is the graph which has as vertices 
the elements of $\mathcal{S}$ and two vertices are adjacent if and only if 
they have a non-empty intersection. Given a graph $G$ and a family $\mathcal{S}$ of subgraphs
of $G$, the intersection graph of $\mathcal{S}$ is the intersection graph
of the family $\{V(H): H \in \mathcal{S}\}$.

In the 1970s, Buneman \cite{buneman1974characterisation}, Gavril
\cite{gavril1974intersection}, and Walter 
\cite{walter1978representations, walter1972phdchordal}, proved independently that chordal graphs 
are exactly the intersection graphs of subtrees in trees.
Let $H$ be a chordal graph. A tree $T$ is a \defin{representation tree}
of $H$ if there exists a function $\beta: V(T) \rightarrow 2^{V(H)}$ 
such that for every $v\in V(H)$, the subgraph $T[\beta^{-1}(v)]$ of $T$ is connected, 
and $H$ is isomorphic to the intersection graph of the family $\{\beta^{-1}(v):v\in V(H)\}$.
In this case we call the pair $(T,\beta)$ a \defin{representation} of $H$.
By the aforementioned characterization of chordal graphs, it follows that every chordal graph
has a representation.
In \autoref{sec:median}, we prove a characterization of graphs of chordality at most $k$
which generalizes the above characterization of chordal graphs. 
We continue with some definitions before we state the main result of \autoref{sec:median}.

An \defin{interval graph} is any graph which is isomorphic
to the intersection graph of a family of intervals on the real line. We denote the hereditary class
of interval graphs by $\mathcal{I}$. 
It is easy to see that the intersection graphs of subpaths in paths are exactly the interval graphs,
and thus every interval graph is also chordal.

Let $G$ be a graph. A \defin{chordal completion} (respectively \defin{interval completion}) of $G$
is a supergraph of $G$ on the same vertex which is chordal (respectively interval).
Since every complete graph is an interval graph, it follows that every graph has an interval 
and thus a chordal completion.

A \defin{tree-decomposition} of $G$ is a representation $(T,\beta)$ of a chordal completion $H$
of $G$. Fix a chordal completion $H$ of $G$ and a representation $(T,\beta)$ of $H$. 
For every $t\in V(T)$, we call the set $\beta(t)$ the \defin{bag} of $t$. 
It is easy to see that every bag is a clique of $H$ and that every clique of $H$
is contained in a bag of $T$.
We say that $(T,\beta)$
is a \defin{complete tree-decomposition} of $G$ if for every $t\in V(T)$, 
the set $\beta(t)$ is a clique of $G$.
If $T$ is a path, then $H$ is an interval completion of $G$ and 
we call tree-decomposition $(T,\beta)$ a \defin{path-decomposition} of $G$.
It is easy to see that a graph has a complete tree-decomposition (respectively complete path-decomposition)
if and only if it is chordal (respectively interval). The \defin{width} of a tree-decomposition is the 
the clique number of the corresponding chordal completion minus one\footnote{The "minus one" in the 
definition of the width serves so that trees have tree-width one.}.
The \defin{tree-width} (respectively \defin{path-width}) of $G$, 
denoted by $\mathsf{tw}(G)$ (respectively $\mathsf{pw}(G)$)
is the minimum width of a tree-decomposition (respectively path-decomposition) of $G$.
That is, $\mathsf{tw}(G):=\min\{\omega(H)-1: H \text{ is a chordal completion of } G\}$, 
and $\mathsf{pw}(G):=\min\{\omega(I)-1: I \text{ is an interval completion of } G\}$.
A tree-decomposition \defin{separates a non-edge} $e$ if $e$ is a non-edge of the 
chordal completion which corresponds to this tree-decomposition. 
Let $\mathcal{T}$ be a family of tree-decompositions of $G$. We say that $\mathcal{T}$
is a \defin{non-edge-separating} family of tree-decompositions if for every non-edge $e$ of $G$, 
there exists a tree-decomposition in $\mathcal{T}$ which 
separates $e$.

Below is the main result of \autoref{sec:median}, 
we postpone some definitions for \autoref{sec:median}. 

\begin{theorem}
\label{thm:chordality.all.intro}
Let $G$ be a graph and $k$ be a positive integer. Then the following are equivalent:
\begin{enumerate}[topsep=0.25cm,itemsep=-0.4cm,partopsep=0.5cm,parsep=0.5cm]
    \item The graph $G$ has chordality $k$.
    \item The minimum size of a non-edge-separating family 
    of tree-decompositions of $G$ is $k$.
    \item $k$ is the minimum integer such that the graph $G$ is the intersection graph of 
    a family of convex subgraphs of the Cartesian product of $k$ trees.
    \item $k$ is the minimum integer such that the graph $G$ is the intersection graph of 
    a family of convex subgraphs of a median graph of tree-dimension $k$.
    \item The graph $G$ has tree-median-dimension $k$.
\end{enumerate}
\end{theorem}

In \autoref{sec:chi_boundedness} we focus on the chromatic number of graphs of bounded chordality.

For a positive integer $k$ we denote by $[k]$ the set of integers $\{1,\ldots,k\}$.
A \defin{$k$-coloring} of $G$ is a function $ f\colon V(G)\rightarrow [k]$ such that 
for every $i\in [k]$ we have that $f^{-1}(i)$ is a stable set. 
A graph is \defin{$k$-colorable} if it admits a $k$-coloring, and the 
\defin{chromatic number} of a graph $G$, denoted by $\chi(G)$, is the minimum integer $k$, for which
$G$ is $k$-colorable.

It is immediate that for every graph $G$ we have $\chi(G) \geq \omega(G)$, and it is easy to see that
there are graphs $G$ for which we have $\chi(G)>\omega(G)$
(for example odd cycles).
Moreover, the gap between the chromatic number and the clique number can be arbitrarily large. 
Indeed, Tutte \cite{descartes1947three, descartes1954solution} first proved in the 1940s 
that there exist triangle-free graphs of arbitrarily large chromatic number 
(for other such constructions see also \cite{burling1965coloring, mycielski1955coloriage, zykov1949some}).
Thus, in general, the chromatic number is not upper-bounded by a function of the clique number. 

A graph $G$ is \defin{perfect} if every induced subgraph $H$ of $G$ satisfies $\chi(H)=\omega(H)$.
Berge \cite{berge1960problemes} proved in 1960 that chordal graphs are perfect.
What can we say for the connection between $\chi$ and $\omega$ for graphs of bounded chordality?

In his seminal paper "Problems from the world surrounding perfect graphs", 
Gy\'arf\'as \cite{gyarfas1985problems} introduced the $\chi$-bounded
graph classes as "natural extensions of the world of perfect graphs".
We say that a hereditary class $\mathcal{A}$ is \defin{$\chi$-bounded} if there exists a function 
$ f\colon \mathbb{N}\rightarrow \mathbb{R}$ such that for every graph $G\in \mathcal{A}$, we have
$\chi(G) \leq f(\omega(G))$. Such a function $f$ is called a \defin{$\chi$-bounding function}
for $\mathcal{A}$. For more on $\chi$-boundedness we refer the readers to the surveys
of Scott and Seymour \cite{scott2020survey}, and Scott \cite{scott2022graphs}. 
The examples of triangle-free graphs of arbitrarily large
chromatic number that we mention above imply that the class of all graphs is not $\chi$-bounded.

A natural direction of research on $\chi$-boundedness 
is to consider operations that we can apply among graphs of different classes 
in order to obtain new classes of graphs, and study (from the perspective of $\chi$-boundedness) 
graph classes which are obtained via this way from $\chi$-bounded classes.

Gy\'arf\'as \cite[Section 5]{gyarfas1985problems} considered graph intersections and 
graph unions of $\chi$-bounded graph classes from the perspective of $\chi$-boundedness. Graph unions of $\chi$-bounded graph classes are $\chi$-bounded\footnote{It is easy to observe that for any two graphs $G_{1}$ and $G_{2}$, we have $\chi(G_{1}\cup G_{2})\leq \chi(G_{1})\chi(G_{2})$ and that $\omega(G_{1}\cup G_{2}) \geq \max\{\omega(G_{1}), \omega(G_{2})\}$. Thus, if for each $i\in [k]$ we have that $f_{i}$ is a $\chi$-bounding function for a class $\mathcal{G}_{i}$, then $f:=\prod_{i\in[k]}f_{i}$ is a $\chi$-bounding function for the class $\biggcup_{i\in [k]} \mathcal{G}_{i}$.}. The situation with intersections of graphs is different. In an upcoming paper two of us, in joint work with Rimma H\"am\"al\"ainen and Hidde Koerts, study further the interplay between graph intersections and $\chi$-boundedness.

Since interval graphs are chordal, it follows that they are perfect as well.
Following \cite{roberts1969boxicity}, we define the \defin{boxicity} of $G$
to be the minimum integer $k$ such that $G$ is isomorphic to the 
intersection graph of a family of axis-aligned boxes in $\mathbb{R}^{k}$.
We denote the boxicity of a graph $G$ by $\mathsf{box}(G)$.
It easy to see that the boxicity $k$ of a graph is equal to its
intersection dimension with respect to the class of interval graphs.

In 1965, in his Ph.D. thesis \cite{burling1965coloring} Burling introduced
a sequence $\{B_{k}\}_{k\geq 1}$ of families of axis-aligned boxes in $\mathbb{R}^{3}$ 
such that for each $k$ the intersection graph of $B_{k}$ is triangle-free and has chromatic number 
at least $k$.
Thus, for every $k\geq 3$ the class of all graphs of boxicity at most $k$, that is, the class
$\biggcap_{i\in [k]}\mathcal{I}$, is not $\chi$-bounded.
Hence, for every $k\geq 3$ the class of graphs of chordality at most $k$ is not 
$\chi$-bounded. 

What about the class $\mathcal{C}\gcap \mathcal{C}$?
Asplund and Gr\"unbaum \cite{asplund1960coloring}, 
in one of the first results which provides an upper bound of the chromatic number
in terms of the clique number for a class of graphs,  proved in 1960 that every intersection graph 
of axis-aligned rectangles in the plane 
with clique number $\omega$ is $\mathcal{O}(\omega^{2})$-colorable. 
Hence the class $\mathcal{I}\gcap \mathcal{I}$ is $\chi$-bounded 
(see also \cite{chalermsook2021coloring} for a better
$\chi$-bounding function). 

Since the class $\mathcal{I}\gcap \mathcal{I}$ is $\chi$-bounded it is natural 
to ask whether any proper superclasses of this class are $\chi$-bounded as well.
Gy\'arf\'as, asked the following question:

\begin{problem}[{Gy\'arf\'as, \cite[Problem 5.7]{gyarfas1985problems}}]
\label{prob:gyarfas}
   Is the class $\mathcal{C}\gcap \mathcal{C}$ $\chi$-bounded? In particular,
   is $\mathcal{C}\gcap\mathcal{I}$ $\chi$-bounded?
\end{problem}

In \autoref{subsec:Burling} we discuss a result of
Felsner, Joret, Micek, Trotter and Wiechert \cite{felsner2017burling}
which implies that Burling graphs are contained in $\mathcal{C}\gcap \mathcal{I}$,
and thus that the answer to Gy\'arf\'as' question is negative.

 In the rest of \autoref{sec:chi_boundedness} we consider two families of subclasses of the class
$\mathcal{C}\gcap \mathcal{C}$, which we prove are $\chi$-bounded. 
In \autoref{sub:pathwidth} we prove the following:

\begin{theorem}
\label{thm:pathwidth.intro}
Let $k_{1}$ and $k_{2}$ be positive integers, and let $G_{1}$ and $G_{2}$
be chordal graphs such that for each $i\in [2]$ the graph $G_{i}$ has a representation $(T_{i},\beta_{i})$, 
where $\mathsf{pw}(T_{i})\leq k_{i}$.
If $G$ is a graph such that $G=G_{1}\cap G_{2}$, then $G$ is 
$\mathcal{O}(\omega(G)\log(\omega(G)))(k_{1}+1)(k_{2}+2)$-colorable.
\end{theorem}

We remark that each of the classes which satisfies the assumptions of \autoref{thm:pathwidth.intro}
is a proper superclass of $\mathcal{I}\gcap \mathcal{I}$.

Let $u$ and $v$ be two vertices of a graph $G$. Then their \defin{distance}, which we denote by $d(u,v)$, is the length of a shortest $(u,v)$-path in $G$. A \defin{rooted tree} is a tree $T$ with one fixed vertex $r\in V(T)$ which we call the \defin{root} of $T$. The \defin{height} of a rooted tree $T$ with root $r$ is $h(T,r) := \max\{d(r, t) : t \in V(T)\}$. The \defin{radius} of a tree $T$, which we denote by $\mathsf{rad}(T)$, is the nonnegative integer $\min\{h(T,r): r \in V(T)\}$. In \autoref{sub:radius} we prove the following:

\begin{theorem}
Let $k$ be a positive integer, and let $G_{1}$ and $G_{2}$ be chordal graphs such that
the graph $G_{1}$ has a representation $(T_{1},\beta_{1})$ where $\mathsf{rad}(T_{1})\leq k$.
If $G$ is a graph such that $G=G_{1}\cap G_{2}$, then $\chi(G) \leq k\cdot \omega(G)$.
\end{theorem}

In \autoref{sec:recognition}, we consider the recognition problem for the class of graphs 
of chordality at most $k$. The \defin{$k$-\textsc{Chordality Problem}} is the following:
Given a graph $G$ as an input, decide whether or 
not $\mathsf{chor}(G)\leq k$. We prove the following:

\begin{theorem}
\label{thm.hardness.of.chordality.intro}
For every $k\geq 3$, the $k$-\textsc{Chordality Problem} is $\NP$-complete.
\end{theorem}

Since chordal graphs can be recognized efficiently 
(see, for example, \cite{gavril1975algorithm, lueker1974structured}), 
the only open case in order to fully classify the complexity of the $k$-\textsc{Chordality Problem},
is the case $k=2$. In an upcoming paper, in joint work with Therese Biedl and Taite LaGrange,
we prove that the $2$-\textsc{Chordality Problem} is $\NP$-complete as well.

\section{A characterization of the graphs of chordality \texorpdfstring{$k$}{TEXT}}
\label{sec:median}

In this section we prove \autoref{thm:chordality.all.intro} which provides different characterizations
of graphs of bounded chordality. We also point out how our proof of \autoref{thm:chordality.all.intro}
can be adapted (in a straight-forward way) to provided analogous characterizations 
for graphs of bounded boxicity.

The key notion that we use for the proof of \autoref{thm:chordality.all.intro}
is that of the {\em tree-median-dimension} of a graph, introduced by Stavropoulos \cite{stavropoulos2016graph}, which we prove is equivalent to chordality.
We introduce the tree-median-dimension of a graph in \autoref{sub:tree-median-dim}.
We first need some definitions.

Let $G$ be a graph. For $u,v \in V(G)$, a \defin{$(u,v)$-geodesic} is a shortest $(u,v)$-path.
We denote by $d(u,v)$ the distance of $u$ and $v$. We also 
denote by $I(u,v)$ the set of all vertices of $G$ which lie in a 
$(u,v)$-geodesic, that is, $I(u, v) := \{x \in V (G) \mid  d(u, v) = d(u, x) + d(x, v)\}$. 
Given three distinct vertices $u,v,w \in V(G)$ we denote by $I(u,v,w)$ the set $I(u,v)\cap I(u,w) \cap I(v,w)$.

A graph $M$ is a \defin{median graph} if it is connected and for every choice of 
three distinct vertices $u,v,w \in V(M)$, there exists a vertex $x$ with the property
that $I(u,v,w) = \{x\}$. 
In this case, the vertex $x$ is called the \defin{median} of $u,v,w$.
For three distinct vertices $u,v,w \in V(G)$, we denote their median vertex by $\mathsf{median}(u,v,w)$.
It is immediate that trees are median graphs.

Given two graphs $G$ and $H$, their \defin{Cartesian product} is the graph $G\cartprod H:=(V,E)$
where $V:= V(G)\times V(H)$ and $\{(v_{1},h_{1}),(v_{2},h_{2})\} \in E$ if and only if 
$v_{1}=v_{2}$ and $h_{1}h_{2}\in E(H)$, or $h_{1}=h_{2}$ and $v_{1}v_{2}\in E(G)$.
A graph is $G$ \defin{isometrically embeddable} into a graph $H$ if there exists a map 
$\phi: V(G) \rightarrow V(H)$ such that for every $u,v\in V(G)$ we have 
$d_{G}(u,v)=d_{H}(\phi(u),\phi(v))$. In this case we write $G \hookrightarrow H$
and we call the map $\phi$ an \defin{isometric embedding}.
The \defin{tree-dimension (respectively \defin{path-dimension})} of a graph $G$, denoted by $\mathsf{td}(G)$ 
(respectively $\mathsf{pd}(G)$) is the minimum $k$ such that $G$ has an isometric 
embedding into the Cartesian product of $k$ trees (respectively paths) if such an embedding exists, and infinite otherwise.

For every positive integer $n$ the \defin{hypercube} $Q_{n}$ is a graph isomorphic to the Cartesian product 
of $n$ copies of $K_{2}$. A \defin{partial cube} is a graph which is isometrically embeddable
into a hypercube. 
Median graphs form a proper subclass of partial cubes (see, for example,
\cite[Theorem 5.75]{ovchinnikov2011graphs}). 
Hence, both the tree-dimension and the path-dimension of every median graph are finite.
We observe that for any graph $G$, we have $\mathsf{td}(G) \leq \mathsf{pd}(G)$.

We say that a set $S\subseteq V(G)$ is \defin{geodesically convex} or simply \defin{convex} 
if for every $u,v \in S$, we have $I(u,v) \subseteq S$. We remark that 
if $G$ is a connected graph and $S\subseteq V(G)$ is a convex set, then $G[S]$ is connected.
A subgraph $H$ of $G$ is a \defin{convex subgraph} if the set $V(H)$ is convex.

\subsection{The tree-median-dimension of a graph}
\label{sub:tree-median-dim}

In his Ph.D. thesis, Stavropoulos \cite{stavropoulos2016graph} introduced median-decompositions of graphs,
and a variant of those, $k$-median-decompositions. We find it more convenient 
for the context of this paper to use the term $k$-tree-median-decomposition for 
the notion of $k$-median-decomposition.

Let $k$ be a positive integer. We say that a graph $H$ has \defin{the property $\mathcal{M}_{k}$} 
if $H$ is the intersection graph of a family of convex subgraphs of a median graph of tree-dimension $k$.
A \defin{representation} of a graph $G$ with the property $\mathcal{M}_{k}$ is a pair $(M,\gamma)$, where 
$M$ is a median graph and $\gamma:V(M)\rightarrow 2^{V(G)}$ is a function such that for every $v \in V(G)$,
we have that $M[\gamma^{-1}(v)]$ is a convex subgraph of $M$, and $G$ 
is isomorphic to the intersection graph of the family $\{\gamma^{-1}(v): v\in V(G)\}$.

Let $G$ be a graph. A \defin{$k$-tree-median-completion} of $G$ is a supergraph 
$H$ of $G$ such that $V(H)=V(G)$, 
and $H$ has the property $\mathcal{M}_{k}$. Observe that $G$ always has a $k$-tree-median-completion
since every complete graph has the property $\mathcal{M}_{k}$.
A \defin{$k$-tree-median-decomposition} of $G$ is a representation $(M,\gamma)$ of 
a $k$-tree-median-completion $H$ of $G$. 
For every $x\in V(M)$, we call the set $\gamma(x)$ the \defin{bag} of $x$. We say that $(M,\gamma)$
is a \defin{complete $k$-tree-median-decomposition} of $G$
if every bag of $M$ is a clique of $G$.
Following Stavropoulos \cite{stavropoulos2016graph}, a 
\defin{median-decomposition} of a graph $G$ is a $k$-tree-median-decomposition for some $k$.
Stavropoulos proved the following:

\begin{theorem}[{Stavropoulos, \cite[Theorem 5.12]{stavropoulos2016graph}}]
    \label{thm: stavropoulos.completeMD}
    Every graph $G$ has a complete median-decomposition.
\end{theorem}

We define the \defin{tree-median-dimension} of a graph $G$, denoted by $\mathsf{tmd}(G)$, as the 
minimum integer $k$ such that $G$ has a complete $k$-tree-median-decomposition. By 
\autoref{thm: stavropoulos.completeMD} it follows that the tree-median-dimension is well defined.

By the following theorem we have that every $1$-tree-median-decomposition is 
a tree-decomposition, and vice versa.

\begin{theorem}[Buneman \cite{buneman1974characterisation}, Gavril \cite{gavril1974intersection}, 
and Walter \cite{walter1978representations, walter1972phdchordal}]
\label{thm:chordal.subtrees.character}
A graph $G$ is the intersection graph of subtrees of a tree if and only if $G$ is chordal.
\end{theorem}

We omit the proof of the following proposition as it follows immediately from the corresponding definitions.

\begin{proposition}
\label{prop.completeMD.intersection}
Let $G$  be a graph and $k$ be a positive integer. 
Then $G$ has a complete $k$-tree-median-decomposition 
if and only if $G$ has the property $\mathcal{M}_{k}$.
\end{proposition}

\autoref{corol.chordal.completeTD} and \autoref{cor.md.equiv.intersection} follow immediately by
\autoref{thm:chordal.subtrees.character} and \autoref{prop.completeMD.intersection}.

\begin{corollary}
\label{corol.chordal.completeTD}
A graph $G$ has a complete tree-decomposition if and only if $G$ is chordal.
\end{corollary}

\begin{corollary}
\label{cor.md.equiv.intersection}
Let $G$ be a graph. Then $\mathsf{tmd}(G)=k$ if and only if 
$k$ is the minimum integer for which $G$ has the property $\mathcal{M}_{k}$.
\end{corollary}

\subsection{A characterization of the graphs of chordality \texorpdfstring{$k$}{TEXT}}
\label{sub:charact.chordality}

The main ingredient that we need for the proof of \autoref{thm:chordality.all.intro} is the
following:

\begin{theorem}
\label{thm:chordality.equiv.tmd}
If $G$ is a graph, then the tree-median-dimension of $G$ is equal to its chordality.
\end{theorem}

We begin with the following easy observation about chordality.

\begin{lemma}
\label{lem:chordality.family.of.TDs}
Let $G$ be a graph. Then the chordality of $G$ is equal to the minimum size 
of a non-edge-separating family of tree-decompositions of $G$.
\end{lemma}

\begin{proof}[Proof of \autoref{lem:chordality.family.of.TDs}]
Let $m$ be the minimum size 
of a non-edge-separating family of tree-decompositions of $G$.

Let $k:=\mathsf{chor}(G)$, and let 
$G_{1}, \ldots , G_{k}$ be $k$ chordal graphs such that 
$G=\bigcap_{i\in [k]}G_{i}$. For each $i\in [k]$ let $(T_{i}, \beta_{i})$ 
be the tree-decomposition of $G$ which is obtained from the chordal completion $G_{i}$, and let 
$\mathcal{T}:=\{(T_{i}, \beta_{i}): i \in [k]\}$. 
Let $\{u,v\}$ be a non-edge of $G$. 
Then there exists $i$ such that $\{u,v\}$ is a non-edge of $G_{i}$. Since
the bags of $(T_{i}, \beta_{i})$ are cliques in $G_{i}$, it follows that 
$(T_{i}, \beta_{i})$ separates $\{u,v\}$. Thus $m \leq \mathsf{chor}(G)$.

Let $\mathcal{T}=\{(T_{i}, \beta_{i})\}_{i\in [m]}$ 
be a non-edge-separating family of tree-decompositions of $G$ such that $|\mathcal{T}|$ is minimized.
For each $i\in[k]$, let $G_{i}$ be the chordal completion of $G$ which corresponds to the tree-decomposition 
$(T_{i}, \beta_{i})$. Let $\{u,v\}$ be a non-edge of $G$ and let $(T_{i},\beta_{i})$ be 
the tree-decomposition of $G$ which separates $\{u,v\}$. Then $\{u,v\} \notin E(G_{i})$. 
Thus, $G=\bigcap_{i\in[k]}G_{i}$, and $\mathsf{chor}(G)\leq m$.
\end{proof}  

In light of \autoref{lem:chordality.family.of.TDs}, in order to prove \autoref{thm:chordality.equiv.tmd}, 
it suffices to prove that the
tree-median-dimension of a graph $G$ is equal to the minimum size of a non-edge-separating family of 
tree-decompositions of $G$. 
We begin with the following lemma, whose proof we omit, as it follows immediately 
from the corresponding definitions.

\begin{lemma}
\label{lem.projection.of.shortest.median.path}
Let $M$ be a median graph, and let $T_{1}, \ldots , T_{k}$ be trees such that there
exists an isometric embedding $\phi\colon V(M)\rightarrow T_{1}\cartprod \cdots \cartprod T_{k}$.
Let $a,b \in V(M)$, let $\pi_{i}:V(T_{1} \cartprod \cdots \cartprod T_{k})\rightarrow V(T_{j})$
be the projection to the $i$-th coordinate, and let $Q=x_{1},\ldots, x_{l}$ be a shortest $(a,b)$-path in $M$.
Then the following hold:
\begin{enumerate}[topsep=0.25cm,itemsep=-0.4cm,partopsep=0.5cm,parsep=0.5cm]
    \item Let $i\in [k]$. If $\pi_{i}(\phi(a))=\pi_{i}(\phi(b))=:t_{i}$, 
    then for every $j\in[l]$ we have $\pi_{i}(\phi(x_{j}))=t_{i}$.
    \item For every $i\in[k]$, we have that $W_{i}\coloneqq \pi_{i}(\phi(x_{1})),\ldots, \pi_{i}(\phi(x_{l}))$ is a sequence of vertices of $T_{i}$ which contains exactly the vertices of the 
    $\left(\pi_{i}(\phi(x_{1})), \pi_{i}(\phi(x_{l}))\right)$-path$=:P$ in $T_{i}$, and these 
    vertices appear in $W_{i}$, possibly with repetitions, in the same order as in $P$.
\end{enumerate}
\end{lemma}

\begin{lemma}
\label{lem.md.from.TDs}
Let $G$ be a graph.
If $\mathcal{T}$ is a non-edge-separating family of tree-decompositions of $G$, then the tree-median-dimension of $G$ is at most $|\mathcal{T}|$.
\end{lemma}

\begin{proof}[Proof of \autoref{lem.md.from.TDs}]
Let $\mathcal{T}:=\{(T_{i},\beta_{i})\}_{i\in [k]}$ be a family of $k$ tree-decompositions of $G$ as in the statement of the lemma. We construct a complete $i$-tree-median-decomposition of $G$, with $i\leq k$. Let $M:=T_{1}\cartprod \ldots \cartprod T_{k}$. Then $M$ is a median graph of tree-dimension at most $k$. Let $\gamma: V(M) \rightarrow 2^{V(G)}$ defined as follows: for every $x\in V(M)$, we have $\gamma(x):= \cap_{i\in[k]}\beta_{i}(\pi_{i}(x))$.

We claim that for every $v \in V(G)$, the subgraph $M[\gamma^{-1}(v)]$ of $M$ is convex. Indeed, the claim follows from \autoref{lem.projection.of.shortest.median.path}, and the fact that the Cartesian product of connected graphs is a connected graph.

We claim that for every $x \in V(M)$, the bag $\gamma(x)$ is a clique. Indeed, this follows from the definition of $\gamma$ and the fact that for every non-edge of $G$ there exists $i\in[k]$ such that no bag of the tree-decomposition $(T_{i},\beta_{i})$ contains both $u$ and $v$.

By the above it follows that $(M,\gamma)$ is a complete $i$-tree-median-decomposition of $G$ which witnesses that $\mathsf{tmd}(G)\leq k$.
\end{proof}

In order to complete the proof of \autoref{thm:chordality.equiv.tmd}, it remains to prove that the minimum size of a non-edge-separating family of tree-decompositions of a graph $G$ is upper-bounded by the tree-median-dimension of $G$. To this end we need some preliminary results. We begin with the statement of a theorem of Stavropoulos \cite{stavropoulos2016graph} which states that given a $k$-tree-median-decomposition of a graph $G$, one can obtain a family of $k$ tree-decompositions of $G$ which satisfy certain nice properties. We then show that if we apply this theorem to a complete $k$-tree-median-decomposition, then the family of $k$ tree-decompositions that we get is non-edge-separating.

\begin{theorem}[{Stavropoulos, \cite[Lemma 6.1, Theorem 6.7]{stavropoulos2016graph}}]
\label{thm.stavropoulos.from.md.to.TDs}
Let $G$ be a graph, and let $(M,\gamma)$ be a $k$-tree-median-decomposition of $G$. Then there exists a family $\mathcal{T} = \{(T_{i},\beta_{i})\}_{i\in [k]}$ of $k$ tree-decompositions of $G$ such that: 
\begin{enumerate}[topsep=0.25cm,itemsep=-0.4cm,partopsep=0.5cm,parsep=0.5cm]
    \item There exists an isometric embedding $\phi$ of $M$ to the graph $T_{1} \cartprod \cdots \cartprod T_{k}$.
    
    \item For every $i\in [k]$ and for every $t\in V(T_{i})$, we have $\phi(V(M))\cap \pi_{i}^{-1}(t)\neq \emptyset$, where $\pi_{i}:V(T_{1} \cartprod \cdots \cartprod T_{k})\rightarrow V(T_{j})$ is the projection to the $i$-th coordinate.
    
    \item For every $x \in V(M)$, we have 
    $\gamma(x) = \bigcap_{\pi_{i}(\phi(x)), i\in[k]} \beta_{i}(\pi_{i}(\phi(x)))$.
    
    \item For every $i\in [k]$, and for every $t\in V(T_{i})$, we have $\beta_{i}(t)= \bigcup_{ \{x \in V(M): \pi_{i}(\phi(x))=t\} } \gamma(x)$.
\end{enumerate}
\end{theorem}

Given a set $X$, we say that a family $\mathcal{X}:=\{X_{i}\}_{i\in I}$ of subsets of $X$ satisfies the \defin{Helly property} if for every $I'\subseteq I$ the following holds: if $X_{i}\cap X_{j}\neq \emptyset$ for all $i,j\in I'$, then we have that $\bigcap_{i\in I'}X_{i}\neq \emptyset$. The following is a folklore (see, for example, \cite[Proposition 4.7]{golumbic2004algorithmic}):

\begin{proposition}
\label{helly.for.subtrees}
Every family of subtrees of a tree satisfies the Helly property.
\end{proposition}

\begin{lemma}
\label{obs.when.a.TD.does.seperate.a.non}
    Let $G$ be a graph, let $(T,\beta)$ be a tree-decomposition of $G$, and let $\{u,v\}$ be a non-edge of $G$ such that $(T,\beta)$ does not separate $\{u,v\}$. Let $t_{1}, t_{2} \in V(T)$ be such that $u\in \beta(t_{1})$ and $v\in \beta(t_{2})$, and let $P$ be the $(t_{1}, t_{2})$-path in $T$. Then, there exists $p\in V(P)$ such that $\{u,v\}\subseteq \beta(p)$.
\end{lemma}

\begin{proof}[Proof of \autoref{obs.when.a.TD.does.seperate.a.non}]
Since $(T,\beta)$ does not separate $\{u,v\}$, we have that $\beta^{-1}(u) \cap \beta^{-1}(v) \neq \emptyset$. Moreover, since $u\in \beta(t_{1})$ and $v\in \beta(t_{2})$, we have that $\beta^{-1}(u) \cap V(P) \neq \emptyset$ and $\beta^{-1}(v) \cap V(P) \neq \emptyset$. Hence, by \autoref{helly.for.subtrees}, it follows that $\beta^{-1}(u)\cap \beta^{-1}(v) \cap V(P)\neq \emptyset$.
\end{proof}

We are now ready to prove that the minimum size of a non-edge-separating family of tree-decompositions of a graph $G$ is upper bounded by the tree-median-dimension of $G$.

\begin{lemma}
\label{lem.takeTDs.from.complete.MD}
Let $G$ be a graph, and let $(M,\gamma)$ be a complete $k$-tree-median-decomposition of $G$. 
Let $\mathcal{T} = \{(T_{i},\beta_{i})\}_{i\in [k]}$
be a family of $k$ tree-decompositions of $G$ which satisfies the 
conditions of \autoref{thm.stavropoulos.from.md.to.TDs}.
Then for every non-edge $e$ of $G$, there exists $i\in [k]$ such that $(T_{i},\beta_{i})$
separates $e$.
\end{lemma}

\begin{proof}[Proof of \autoref{lem.takeTDs.from.complete.MD}]
Let us suppose towards a contradiction that the lemma does not hold.
Let $\{u,v\} \in \binom{V(G)}{2}\setminus E(G)$ be such that no tree-decomposition in $\mathcal{T}$ separates
$\{u,v\}$, and let $\phi$ be an isometric embedding of $M$ to the graph 
$T_{1} \cartprod \cdots \cartprod T_{k}$, as 
in the statement of \autoref{thm.stavropoulos.from.md.to.TDs}.

Let $S:=\{(a,b)\in V(M)\times V(M):u\in\gamma(a), v\in\gamma(b)\}$, and for each $j\in [k]$, let $S_{j}:=\{(a,b)\in S: (\forall l\leq j)[\pi_{l}(\phi(a))=\pi_{l}(\phi(b))=:t_{l}\}.$ In what follows we derive the desired contradiction by proving that there exists a vertex of $M$ whose bag, in $(M,\gamma)$, contains both the vertices $u$ and $v$. To this end it suffices to prove that $S_{k}\neq \emptyset$.

\begin{claim}
\label{median.Sk.nonempty}
For each $j\in [k]$, we have $S_{j} \neq \emptyset$.
\end{claim}

\begin{subproof}[Proof of \autoref{median.Sk.nonempty}]
Since $\{u,v\}\subseteq \beta_{1}(t_{1})$, by \autoref{thm.stavropoulos.from.md.to.TDs} (4), there exist not necessarily distinct $a,b\in V(M)$ such that $\pi_{1}(\phi(a))=\pi_{1}(\phi(b))=t_{1}$, $u\in \gamma(a)$, and $v \in \gamma(b)$. Hence $(a,b) \in S_{1}$ and thus $S_{1} \neq \emptyset$. Let $j:=\max\{i\in[k]:(\forall l\leq i)[S_{i}\neq \emptyset]\}$. Since $S_{1}\neq \emptyset$, we have that $j$ is well defined. Let us suppose towards a contradiction that $j<k$, and let $(a,b)\in S_{j}$. For each $i\in[j]$, let $t_{i}\in V(T_{i})$ be such that $\pi_{i}(\phi(a))=\pi_{i}(\phi(b))=t_{i}$. Let $P$ be the $\big( (\pi_{j+1}(\phi(a))), (\pi_{j+1}(\phi(b))) \big)$-path in $T_{j+1}$. By \autoref{obs.when.a.TD.does.seperate.a.non}, it follows that there exists $t\in V(P) \subseteq V(T_{j+1})$ such that $\{u,v\}\subseteq \beta_{j+1}(t)$. Let $t_{j+1}$ be such a vertex, and let $Q$ be a shortest $(a,b)$-path in $M$.

We claim that there exists $z\in V(Q)$ such that for each $i\in [j+1]$ we have
$\pi_{i}(\phi(z))=t_{i}$. Indeed, by \autoref{lem.projection.of.shortest.median.path}, we know that for every vertex $q\in V(Q)$ and for every $i\in [j]$ we have $\pi_{i}(\phi(q))=t_{i}$. Since $t_{j+1}$ lies in $P$, it follows, by \autoref{lem.projection.of.shortest.median.path}, that there exists $z\in V(Q)$ such that $\pi_{j+1}(\phi(z))=t_{j+1}$. Let $z$ be such a vertex. Then $z$ satisfies our claim.

Since $\{u,v\}\subseteq \beta_{j+1}(t_{j+1})$, by \autoref{thm.stavropoulos.from.md.to.TDs} (4), 
it follows that there exist not necessarily distinct vertices $x,y\in V(M)$ such that 
$u\in \gamma(x), v\in \gamma(y)$ and $\pi_{j+1}(\phi(x))=\pi_{j+1}(\phi(y))=t_{j+1}$.

Let $a':=\mathsf{median}(a,x,z)$ and $b':=\mathsf{median}(b,y,z)$. Then, by \autoref{lem.projection.of.shortest.median.path}, we have $\pi_{i}(\phi(a'))=\pi_{i}(\phi(b'))=t_{i}$ for all $i\in [j+1]$. Hence $(a',b')\in S_{j+1}$, and thus $S_{j+1}\neq \emptyset$ which contradicts to the choice of $j$. This concludes the proof of \autoref{median.Sk.nonempty}.
\end{subproof}

Let $(a,b) \in S_{k}$. Then, by the definition of $S$ and \autoref{thm.stavropoulos.from.md.to.TDs}, we have that  $\gamma(a)=\gamma(b)$. In particular, $\{u,v\}\subseteq \gamma(a)$ which is a contradiction.
This concludes the proof of \autoref{lem.takeTDs.from.complete.MD}.
\end{proof}

\autoref{cor.tmd.equiv.tree_decomp} follows 
immediately from \autoref{lem.takeTDs.from.complete.MD} and \autoref{lem.md.from.TDs}.

\begin{corollary}
\label{cor.tmd.equiv.tree_decomp}
    Let $G$ be a graph. 
    Then the tree-median-dimension of $G$ is equal to the minimum size of a non-edge-separating family 
    of tree-decompositions of $G$.
\end{corollary}

Now \autoref{thm:chordality.equiv.tmd}, which states that the tree-median-dimension of a graph is equal to its chordality, follows immediately by \autoref{lem:chordality.family.of.TDs} and \autoref{cor.tmd.equiv.tree_decomp}. 

\autoref{thm:chordality.all.intro}, is an immediate corollary of \autoref{thm:chordality.equiv.tmd}, \autoref{prop.completeMD.intersection}, \autoref{cor.md.equiv.intersection} and \autoref{lem:chordality.family.of.TDs}. We remark that \autoref{thm:chordality.all.intro} generalizes 
\autoref{thm:chordal.subtrees.character} and \autoref{corol.chordal.completeTD}.

\medskip
The notion of $k$-path-median-decomposition can be defined similarly with that of $k$-tree-median-decomposition, by considering completions which are intersection graphs of convex subgraphs of median graphs of path-median-dimension $k$. By modifying the proofs of this section in a trivial way we can derive the following characterizations of boxicity.

\begin{theorem}
\label{thm.boxicity.all}
Let $G$ be a graph and $k$ be a positive integer. Then the following are equivalent:
\begin{enumerate}[topsep=0.25cm,itemsep=-0.4cm,partopsep=0.5cm,parsep=0.5cm]
    \item The graph $G$ has boxicity $k$.
    \item The minimum size of a non-edge-separating family 
    of path-decompositions of $G$ is $k$.
    \item $k$ is the minimum integer such that the graph $G$ is the intersection graph of 
    a family of convex subgraphs of the Cartesian product of $k$ paths.
    \item $k$ is the minimum integer such that the graph $G$ is the intersection graph of 
    a family of convex subgraphs of a median graph of path-dimension $k$.
    \item The graph $G$ has path-median-dimension $k$.
\end{enumerate}
\end{theorem}

\section{Chordality and \texorpdfstring{$\chi$}{TEXT}-boundedness}
\label{sec:chi_boundedness}

We study classes of graphs of bounded chordality from the perspective of $\chi$-boundedness. 

\subsection{The class \texorpdfstring{$\mathcal{C}\gcap \mathcal{I}$}{TEXT} is not 
\texorpdfstring{$\chi$}{TEXT}-bounded}
\label{subsec:Burling}

In \cite{dujmovic2018orthogonal} Dujmovic, Joret, Morin, Norin, and Wood studied graphs 
which have two tree-decompositions such that "each bag of the first decomposition has a bounded
intersection with each bag of the second decomposition", and pointed out a connection of this 
concept with the concept of $\chi$-boundedness.

Following \cite{dujmovic2018orthogonal} we say that two tree-decompositions 
$(T_{1}, \beta_{1})$ and $(T_{2}, \beta_{2})$ of a graph $G$ are \defin{$k$-orthogonal} if 
for every $t_{1}\in T_{1}$ and $t_{2}\in T_{2}$, we have $|\beta_{1}(t_{1})\cap \beta_{2}(t)|\leq k$.

\begin{lemma}
\label{lem:orthogonal.TDs.chordality}
Let $G$ be a graph and $k$ be a positive integer. Then the following hold:
\begin{enumerate}[topsep=0.25cm,itemsep=-0.4cm,partopsep=0.5cm,parsep=0.5cm]
    \item The graph $G$ has two $k$-orthogonal path-decompositions
    if and only if $G$ is a subgraph of a graph $H$ such that $H$ has 
    boxicity at most two, and $\omega(H)\leq k$.
   \item The graph $G$ has two $k$-orthogonal tree-decompositions
    if and only if $G$ is a subgraph of a graph $H$ such that $H$ has 
    chordality at most two, and $\omega(H)\leq k$.
   \item The graph $G$ has a tree-decomposition and a path-decomposition which are $k$-orthogonal
   if and only if $G$ is a subgraph of a graph $H$ such that $H \in \mathcal{C}\gcap \mathcal{I}$, 
   and $\omega(H)\leq k$.
\end{enumerate}
\end{lemma}

\begin{proof}[Proof of \autoref{lem:orthogonal.TDs.chordality}]
  Follows immediately by the corresponding definitions and 
  the facts that every bag of a tree-decomposition is a clique of the corresponding chordal completion,
  and that every clique of a chordal completion is contained in a bag of the corresponding
  tree-decomposition.
\end{proof}

The following is an immediate corollary of \autoref{lem:orthogonal.TDs.chordality}.

\begin{proposition}
\label{prop:orthogonal.chi.bound}
    Let $\mathcal{C}$ be the class of chordal graphs and $\mathcal{I}$ be the class of interval graphs.
    The following hold:
    \begin{enumerate}[topsep=0.25cm,itemsep=-0.4cm,partopsep=0.5cm,parsep=0.5cm]
        \item The class $\mathcal{I}\gcap \mathcal{I}$ is $\chi$-bounded if and only if 
        there exists a function $f\colon \mathbb{N}\rightarrow \mathbb{R}$ such that for every graph $G$
        which has two $k$-orthogonal path-decompositions, we have $\chi(G)\leq f(k)$.
        \item The class $\mathcal{C}\gcap \mathcal{C}$ is $\chi$-bounded if and only if 
        there exists a function $ f\colon \mathbb{N}\rightarrow \mathbb{R}$ such that for every graph $G$
        which has two $k$-orthogonal tree-decompositions, we have $\chi(G)\leq f(k)$.
        \item The class $\mathcal{C}\gcap \mathcal{I}$ is $\chi$-bounded if and only if 
        there exists a function $ f\colon \mathbb{N}\rightarrow \mathbb{R}$ such that for every graph $G$
        which has a tree-decomposition and a path-decomposition which are $k$-orthogonal, 
        we have $\chi(G)\leq f(k)$.
    \end{enumerate}
\end{proposition}

The authors of \cite{dujmovic2018orthogonal} 
posed the following question, for which they conjectured a positive answer.

\begin{problem}[{Dujmovic, Joret, Morin, Norin, and Wood, \cite[Open Problem 3]{dujmovic2018orthogonal}}]
\label{prob:chordality.orthogonal.coloring}
Is there a function $f$ such that every graph $G$ that has two $k$-orthogonal tree-decompositions 
is $f(k)$-colorable?
\end{problem}

By \autoref{prop:orthogonal.chi.bound}, it follows that the above question 
is equivalent to the first part of the question of Gy\'arf\'as that we metioned in the 
\nameref{sec:introduction} (\autoref{prob:gyarfas}), which
asks whether the class of all graphs of chordality at most two is $\chi$-bounded.

In \cite{felsner2017burling} Felsner, Joret, Micek, Trotter and Wiechert, answered \autoref{prob:chordality.orthogonal.coloring} in the negative, and in particular they answered Gy\'arf\'as's question (\autoref{prob:gyarfas}), in the negative. 

Felsner, Joret, Micek, Trotter and Wiechert \cite{felsner2017burling} proved the following, which answers 
in the negative both the questions in \autoref{prob:gyarfas} and \autoref{prob:chordality.orthogonal.coloring}.

\begin{theorem}[{Felsner, Joret, Micek, Trotter and Wiechert, \cite[Theorem 2]{felsner2017burling}}]
\label{th.Burling.decompositions}
For every positive integer $k$, there is a graph with chromatic number at least $k$ which has a tree-decomposition $(T, \beta)$ and a path-decomposition $(P, \gamma)$, which are 2-orthogonal. That is, for every $t \in V(T)$ and for every $p\in V(P)$, we have $|\beta(t) \cap \gamma(p)| \leq 2$.
\end{theorem}

The following is an immediate corollary of \autoref{lem:orthogonal.TDs.chordality}
and \autoref{th.Burling.decompositions}.

\begin{corollary}
\label{cor:Burling.in.chordal.interv}
For every positive integer $k$, there exist a graph $H_{k}\in \mathcal{C}\gcap \mathcal{I}$
such that $H_{k}$ is triangle-free and has chromatic number at least $k$.
\end{corollary}

\begin{corollary}
\label{cor:tree.path.not}
The class $\mathcal{C}\gcap \mathcal{I}$ is not $\chi$-bounded. In particular, since 
$\mathcal{I} \subseteq \mathcal{C}$, it follows that the class of all the graphs of chordality
at most two is not $\chi$-bounded.
\end{corollary}

\subsection{Subclasses of \texorpdfstring{$\mathcal{C}\gcap \mathcal{C}$}{TEXT}: 
When each chordal graph has a representation tree of bounded path-width}
\label{sub:pathwidth}
In \autoref{subsec:Burling}, we saw that the class $\mathcal{C}\gcap \mathcal{I}$ is not $\chi$-bounded. From the characterization of chordal (respectively interval) graphs as intersection graphs of subtrees (respectively subpaths) of trees (respectively paths) that we presented in the \nameref{sec:introduction},
it follows that $\mathcal{I}\gcap \mathcal{I}$ is the subclass of $\mathcal{C}\gcap \mathcal{C}$
in which each of the two chordal graphs in the intersection has a representation tree which is a path.

In this subsection we consider the family of subclasses of $\mathcal{C}\gcap \mathcal{C}$
(and superclasses of $\mathcal{I}\gcap \mathcal{I}$) in which each of 
the two chordal graphs in the intersection has a representation tree
of bounded path-width. We prove that these classes are $\chi$-bounded.

\begin{theorem}
\label{thm.both.trees.bounded.pw}
Let $k_{1}$ and $k_{2}$ be positive integers, and let $G_{1}$ and $G_{2}$ be chordal graphs such that for each $i\in [2]$ the graph $G_{i}$ has a representation $(T_{i},\beta_{i})$, where $\mathsf{pw}(T_{i})\leq k_{i}$. If $G$ is a graph such that $G=G_{1}\cap G_{2}$, then $G$ is $\mathcal{O}(\omega(G)\log(\omega(G)))(k_{1}+1)(k_{2}+1)$-colorable.
\end{theorem}

The main step towards our proof of \autoref{thm.both.trees.bounded.pw} is to prove that
the vertex set of a graph $G$ as in the statement of \autoref{thm.both.trees.bounded.pw} can be partitioned into 
a constant number of sets so that each of these sets induces a graph of boxicity at most two.
Then we use the fact that the class $\mathcal{I}\gcap \mathcal{I}$ 
is $\chi$-bounded and we color each of these induced subgraphs with a different 
palette of colors.

\begin{lemma}
\label{lem:bounded.pw.partition}
Let $k_{1}$ and $k_{2}$ be positive integers, and let $G_{1}$ and $G_{2}$
be chordal graphs such that for each $i\in [2]$ the graph $G_{i}$ has a representation $(T_{i},\beta_{i})$, 
where $\mathsf{pw}(T_{i})\leq k_{i}$.
If $G$ is a graph such that $G=G_{1}\cap G_{2}$, then there exists a partition $\mathcal{P}$ of 
$V(G)$ such that $|\mathcal{P}|\leq (k_{1}+1)(k_{2}+1)$ and for every $V\in\mathcal{P}$, 
the graph $G[V]$ has boxicity at most two.
\end{lemma}

In 2021, Chalermsook and Walczak \cite{chalermsook2021coloring} provided an improvement 
on the upper bound of Asplund and Gr\"unbaum \cite{asplund1960coloring} for the chromatic number
of graphs of boxicity at most two.

\begin{theorem}[Chalermsook and Walczak, \cite{chalermsook2021coloring}]
\label{thm.ref.coloring.rectangles.2021}
Every family of axis-parallel rectangles in the plane with 
clique number $\omega$ is $\mathcal{O}(\omega\log(\omega))$-colorable, 
and an $\mathcal{O}(\omega\log(\omega))$-coloring of it can be computed in polynomial time.
\end{theorem}

Since \autoref{thm.both.trees.bounded.pw} follows immediately by \autoref{lem:bounded.pw.partition} and \autoref{thm.ref.coloring.rectangles.2021}, in order to prove \autoref{thm.both.trees.bounded.pw} it remains to prove \autoref{lem:bounded.pw.partition}. The main observation that we need is that if a graph has path-width at most $k$, then it can decomposed into a family of $k+1$ disjoint subgraphs, each of which is a disjoint union of induced paths.

We first need a result about tree-decompositions. Let $G$ be a graph, $X\subseteq V(G)$, and $u,w$ two vertices of $G$. We say that $X$ \defin{separates} $u$ from $w$ in $G$ if $u$ and $w$ are in different components of the graph $G-X$.

\begin{lemma}[{Robertson and Seymour \cite[(2.4)]{graph.minors.II.TW.1986}}]
\label{lem.TD.induced.seperators}
Let $G$ be a graph, $(T,\beta)$ be a tree-decomposition of $G$, let $\{t_{1}, t_{2}\}$
be an edge of $T$. If $T_{1}$ and $T_{2}$ are the components of $T\setminus \{t_{1}, t_{2}\}$, 
where $t_{1}\in V(T_{1})$ and $t_{2}\in V(T_{2})$, 
then $\beta(t_{1}) \cap \beta(t_{2})$ separates $V_{1}:=\bigcup_{t\in V(T_{1})}\beta(t)$
from $V_{2}:=\bigcup_{t\in V(T_{2})}\beta(t)$ in $G$.
\end{lemma}

\begin{lemma}
\label{lem.deleting.a.path}
Let $G$ be a connected graph and let $k$ be a positive integer. If $G$ has path-width at most $k$, then there exists an induced path $Q$ which is a subgraph of $G$ such that $G\setminus V(Q)$ has path-width at most $k-1$.
\end{lemma}

\begin{proof}[Proof of \autoref{lem.deleting.a.path}]
Consider a path-decomposition $(P,\beta)$ of $G$ which realizes its path-width. Let $p_{1},\ldots,p_{l}$ be the elements of $V(P)$ enumerated in the order that they appear in $P$. Let $v_{1}\in V(G)\cap \beta(p_{1})$ and $v_{l} \in V(G)\cap \beta(p_{l})$, and let $Q$ be an induced $(v_{1}, v_{l})$-path in $G$. We define the function $\beta':V(P)\rightarrow 2^{V(G)}$ as follows: for every $p\in V(p)$ we have $\beta'(p):=\beta(p)\setminus V(Q)$. Then $(P,\beta')$ is a path-decomposition of the graph $G\setminus V(Q)$. Moreover, by \autoref{lem.TD.induced.seperators}, it follows that for each $i\in [l]$ we have $V(Q)\cap \beta(p_{i}) \neq \emptyset$. Thus, the width of $(P,\beta')$ is at most $k-1$.
\end{proof}

\begin{corollary}
\label{corol.path.decomp.for.proof}
Let $k$ be a positive integer. If $G$ is a graph of path-width at most $k$, 
then there exist (possibly null) induced subgraphs $P_{1}, \ldots , P_{k+1}$ of $G$ such that the following hold:
\begin{enumerate}[topsep=0.25cm,itemsep=-0.4cm,partopsep=0.5cm,parsep=0.5cm]
\item Every component of the graph $P_{1}$ is a path.
\item For each $i\in [2,k+1]$, $P_{i}$ is an induced subgraph of $G\setminus (V(P_{1}) \cup \ldots \cup V(P_{i-1}))$, and every component of $G \setminus \cup_{j<i}V(P_{j})$ contains exactly one component of $P_{i}$.
\item $V(G)= \bigcup_{i \in [k+1]} V(P_{i})$.
\end{enumerate}
\end{corollary}

\begin{proof}[Proof of \autoref{corol.path.decomp.for.proof}]
We prove the statement by induction on $k$. If the graph $G$ has path-width equal to one, then $G$ is the disjoint union of paths, and letting $P_{1}:= G$ we see that the statement of \autoref{corol.path.decomp.for.proof} holds.

Let $k>1$ and suppose that the statement of 
\autoref{corol.path.decomp.for.proof} holds for every positive integer $k'<k$.
Let $C_{1}, \ldots , C_{l}$ be the connected components of $G$. For each $j \in [l]$, we have that 
$C_{j}$ is a connected graph of path-width at most $k$. Let
$P_{1}^{j}$ be a subgraph of $C_{j}$ which is a path as in the statement of 
\autoref{lem.deleting.a.path}. Let $P_{1}: = \cup_{j\in [l]} P_{1}^{j}$.
Consider the graph $G':=G\setminus V(P_{1})$ which, by \autoref{lem.deleting.a.path}, has
path-width at most $k':=k-1<k$. Then, by applying the induction hypothesis to the graph 
$G'$, we obtain subgraphs $P_{2}, \ldots , P_{k+1}$ of $G'$ such that the subgraphs 
$P_{1}, \ldots , P_{k+1}$ of $G$ satisfy the statement of 
\autoref{corol.path.decomp.for.proof}.
\end{proof}

We are now ready to prove \autoref{lem:bounded.pw.partition}.

\begin{proof}[Proof of \autoref{lem:bounded.pw.partition}]
Let $P_{1}, \ldots, P_{k_{1}+1}$ and $Q_{1}, \ldots ,Q_{k_{2}+1}$ 
be subgraphs of $T_{1}$ and $T_{2}$ respectively, 
chosen as in \autoref{corol.path.decomp.for.proof}.

Let $X$ be a subtree of $T_{1}$. We define the level of $X$, denoted by $L_{1}(X)$, as follows: 
$$L_{1}(X) := \min\{i\in[k_{1} + 1]:V(X)\cap V(P_{i})\neq \emptyset\}.$$
Similarly, we define the level of a subtree $X$ of $T_{2}$ as follows:
$$L_{2}(X):= \min\{i\in[k_{2} + 1]:V(X)\cap V(Q_{i})\neq \emptyset\}.$$

\begin{claim}
\label{cl.1.thm.both.trees.bounded.pw}
Let $X$ and $Y$ be subtrees of $T_{1}$ such that $L_{1}(X)=L_{1}(Y)=l$. Then the following hold:
\begin{enumerate}
    \item Both $X\cap P_{l}$ and $Y\cap P_{l}$ are paths; and
    \item $V(X)\cap V(Y)\neq \emptyset$ if and only if $V(X)\cap V(Y)\cap V(P_{l}) \neq \emptyset$.
\end{enumerate}  
Similarly for subtrees of $T_{2}$.
\end{claim}

\begin{subproof}[Proof of \autoref{cl.1.thm.both.trees.bounded.pw}]
We prove the claim for $T_{1}$; the proof for $T_{2}$ is identical. Since $L_{1}(X)=l$, we have that $V(X)\cap (V(P_{1}) \cup \ldots \cup V(P_{l-1})) = \emptyset$. Thus $X$ is contained in a connected component of the forest $T \setminus (V(P_{1}) \cup \ldots \cup V(P_{l-1}))$. Let $C$ be this component. By \autoref{corol.path.decomp.for.proof}, we have that $Z:=P_{l}\cap C$ is a path, and thus $X\cap P_{l}$ is a path as well. With identical arguments we get that $Y\cap P_{l}$ is a path.

For the second statement of our claim: The reverse implication is immediate. For the forward implication: Since $V(X)\cap V(Y)\neq \emptyset$, both the subtrees $X$ and $Y$ are contained in the same connected component of the forest $T \setminus (V(P_{1}) \cup \ldots \cup V(P_{l-1}))$. Let $C$ be this component. By \autoref{corol.path.decomp.for.proof}, we have that $Z:=P_{l}\cap C$ is a path. Consider the tree $C$ and its family of subtrees $\{X,Y,Z\}$. Since $V(X)\cap V(Z) \neq \emptyset$, $V(Y)\cap V(Z) \neq \emptyset$ and $V(X)\cap V(Y) \neq \emptyset$, by \autoref{helly.for.subtrees}, it follows that $V(X)\cap V(Y) \cap V(Z) \neq \emptyset$. In particular $V(X)\cap V(Y) \cap V(P_{l}) \neq \emptyset$. This concludes the proof of \autoref{cl.1.thm.both.trees.bounded.pw}.
\end{subproof}

In what follows in this proof, for every $v\in V(G)$ and $i\in[2]$, we denote by $T_{i}^{v}$ the subtree 
$T_{i}[\beta_{i}(v)]$ of $T_{i}$.
For each $i\in [k_{1}+1]$ and for each $j\in [k_{2}+1]$, we define a subset of $V(G)$
as follows: $$V_{i,j}:=\{v\in V(G):L_{1}(T_{1}^{v})=i \text{ and } L_{2}(T_{2}^{v})=j\}.$$
Let $\mathcal{P}:=\{V_{i,j}\}_{i \in [k_{1}+1], j\in [k_{2}+1]}$ 
and observe that $\mathcal{P}$ is a partition of $V(G)$.

\begin{claim}
\label{claim.interval}
For each $i\in [k_{1}+1]$ the graph $G_{1}[\bigcup_{j\in [k_{2}+1]}V_{i,j}]$ is an interval graph. Similarly for each $j\in [k_{2}+1]$, and $G_{2}[\bigcup_{i\in [k_{1}+1]}V_{i,j}]$.
\end{claim}

\begin{subproof}[Proof of \autoref{claim.interval}]
We prove the claim for $G_{1}$; the proof for $G_{2}$ is identical. For each vertex $v \in G_{1}[\bigcup_{j\in [k_{2}+1]}V_{i,j}]$, let $P_{i}^{v}:= P_{i}\cap T_{1}^{v}$. Then, by \autoref{cl.1.thm.both.trees.bounded.pw}, we have that $P_{i}^{v}$ is a path. Let $u,v \in G_{1}[\bigcup_{j\in [k_{2}+1]}V_{i,j}]$. Then, by \autoref{cl.1.thm.both.trees.bounded.pw}, we have that $u$ is adjacent to $v$ if and only if $P_{i}^{u} \cap P_{i}^{v} \neq \emptyset$. 

Hence, the graph $G_{1}[\bigcup_{j\in [k_{2}+1]}V_{i,j}]$
is the intersection graph of the family $\{P_{i}^{v} :v \in G_{1}[\bigcup_{j\in [k_{2}+1]}V_{i,j}]\}$
of subpaths $P_{i}$. Since $P_{i}$ is the disjoint union of paths
it follows that every component of $G_{1}[\bigcup_{j\in [k_{2}+1]}V_{i,j}]$ is an interval graph, and so $G_{1}[\bigcup_{j\in [k_{2}+1]}V_{i,j}]$ is an interval graph as well.
\end{subproof}

Let $i\in [k_{1}+1]$ and $j\in [k_{2}+1]$. Then, by \autoref{claim.interval}, we have $G[V_{i,j}]\in \mathcal{I}\gcap \mathcal{I}$. Hence $\mathcal{P}$ is the desired partition.
\end{proof}

\subsection{Subclasses of \texorpdfstring{$\mathcal{C}\gcap \mathcal{C}$}{TEXT}: 
When at least one chordal graph has a representation tree of bounded radius}
\label{sub:radius}

For each positive integer $k$ we consider the subclass of $\mathcal{C}\gcap \mathcal{C}$ in 
which one of the two chordal graphs in the intersection has 
a representation tree of radius at most $k$, and we prove that this class is $\chi$-bounded.

\begin{theorem}
\label{thm.bounded.radius}
Let $k$ be a positive integer, and let $G_{1}$ and $G_{2}$ be chordal graphs such that
the graph $G_{1}$ has a representation $(T_{1},\beta_{1})$ where $\mathsf{rad}(T_{1})\leq k$.
If $G$ is a graph such that $G=G_{1}\cap G_{2}$, then $\chi(G) \leq k\cdot \omega(G)$.
\end{theorem}

The main observation that we need for the proof of \autoref{thm.bounded.radius} is 
the following:

\begin{lemma}
\label{lem:radius.partition}
Let $G$ be a chordal graph and $k$ be a positive integer. 
If $G$ has a representation $(T,\beta)$ such that $\mathsf{rad}(T) \leq k$, 
then there exists a partition $\mathcal{P}$ of $V(G)$ such that $|\mathcal{P}|\leq k$
and for each $V\in \mathcal{P}$ we have that $G[V]$ is a disjoint union of complete graphs.
\end{lemma}

We show how \autoref{thm.bounded.radius} follows from \autoref{lem:radius.partition}.

\begin{proof}[Proof of \autoref{thm.bounded.radius} assuming \autoref{lem:radius.partition}]
Let $G$ be a graph as in the statement of \autoref{thm.bounded.radius}, and let 
$\mathcal{P}$ be 
a partition of $V(G_{1})$ as in the statement of \autoref{lem:radius.partition}.

We claim that for each $V\in \mathcal{P}$, we have $\chi(G[V]) \leq \omega(G)$.
Indeed, let $V\in \mathcal{P}$. Then the graph $G_{1}[V]$ is a disjoint union of complete graphs. 
Hence, the graph $G[V] = G_{1}[V] \cap G_{2}[V]$ 
is the intersection of a chordal graph with a disjoint union of cliques, 
and thus a chordal graph. Hence, $\chi(G[V]) \leq \omega(G[V]) \leq \omega(G)$.

For each $V\in \mathcal{P}$, we can color the graph $G[V]$ with a different palette of $\omega(G)$
colors, and obtain a $(k\cdot \omega(G))$-coloring of $G$. Hence $\chi(G) \leq k\cdot \omega(G)$.
\end{proof}

It remains to prove \autoref{lem:radius.partition}.

\begin{proof}[Proof of \autoref{lem:radius.partition}]
Let $r$ be a vertex of $T$ which, chosen as a root, realizes the radius of $T$.
For each vertex $v \in V(G)$, we denote by $T^{v}$ the subtree $T[\beta(v)]$ of $T$.
Furthermore, for each subtree $X$ of $T$, we denote by 
$L(X)$ the value $\min\{d(r,x):x\in V(X)\}$, 
and by $r(X)$ the root of $X$, which is the unique element of the set $\argmin_{x\in V(X)} d(r,x)$.
We refer to the value $L(X)$ as the level of $X$.

Let $S:= \{T^{v}: v\in V(G)\}$, and for each $i\in [k]$, let $L_{i}:= \{X\in S: L(X) = i\}$. 
Observe that, since $T$ has radius at most $k$, 
we have that $\{L_{i}\}_{i\in [k]}$ is a partition of $S$.

The main observation that we need is that two subtrees $X$ and $Y$ of the same level intersect 
if and only if they have the same root (and no other common vertex). 

Thus, for each level $i\in[k]$, the relation of intersection of subtrees is an equivalence relation 
in $L_{i}$, and the corresponding induced subgraph of $G$ is a disjoint union of complete graphs.

For each $i\in [k]$, let $V_{i}:=\{v\in V(G):T^{v}\in L_{i}\}$. 
Then $\mathcal{P}:=\{V_{i}:i\in[k] \text{ and } V_{i}\neq \emptyset\}$ 
is the desired partition of $V(G)$.
\end{proof}

\section{The \texorpdfstring{$k$}{TEXT}-\textsc{Chordality Problem} is \texorpdfstring{$\NP$}{TEXT}-complete for \texorpdfstring{$k\geq 3$}{TEXT}}
\label{sec:recognition}

We recall from the \autoref{sec:introduction} that for 
a fixed positive integer $k$, the $k$-\textsc{Chordality Problem} is the following:
Given a graph $G$ as an input, decide whether $\mathsf{chor}(G)\leq k$.
In this section we study the computational complexity of this problem.
For $k=1$, the $1$-\textsc{Chordality Problem} is to decide whether a given
graph is chordal, and there exists a polynomial-time algorithm for this problem 
(see, for example, \cite{gavril1975algorithm, leuker1976algorithmic}). 
In this section we prove \autoref{thm.hardness.of.chordality.intro} which we restate here.

\begin{theorem}
\label{thm.hardness.of.chordality}
For every $k\geq 3$, the $k$-\textsc{Chordality Problem} is $\NP$-complete.
\end{theorem}
In an upcoming paper, in joint work with Therese Biedl and Taite LaGrange,
using different techniques we prove that the $2$-\textsc{Chordality Problem} is $\NP$-complete as well.

For a fixed positive integer $k$ the $k$-\textsc{Coloring Problem} is the following:
Given a graph $G$ as an input, decide whether $G$ has a $k$-coloring.

\begin{theorem}[{Karp, \cite[Main Theorem]{karp1972reducibility}}]
\label{thm.hardeness.of.coloring}
For every $k\geq 3$, the $k$-\textsc{Coloring Problem} is $\NP$-complete.
\end{theorem}

We immediately see that for every positive integer $k$, the $k$-\textsc{Chordality Problem} is in $\NP$. We prove \autoref{thm.hardness.of.chordality} by proving a polynomial-time reduction of the $k$-\textsc{Coloring Problem} to the $k$-\textsc{Chordality Problem}. We first state some preliminary definitions and results.

\begin{theorem}[{McKee and Scheinerman, \cite[Corollary 4]{mckee1993chordality}}]
\label{thm.chordality.and.chromatic.number}
    Let $G$ be a graph. Then $\mathsf{chor}(G) \leq \chi(G)$.
\end{theorem}

Given two graphs $G$ and $H$, the \defin{lexicographic product} of $G$ with $H$, denoted by $G\cdot H$, is the graph which has as vertices the elements of the set $V(G)\times V(H)$, and where two vertices $(x_{1},y_{1})$ and $(x_{2},y_{2})$ are adjacent if and only if $\{x_{1}, x_{2}\}\in E(G)$, or $x_{1}=x_{2}$ and $\{y_{1}, y_{2}\}\in E(H)$. The graph $G\cdot H$ can be though as the graph that we obtain if in $G$ we "substitute" a copy of $H$ for each vertex of $G$.

\begin{theorem}[{Geller and Stahl, \cite[Theorem 3]{geller1975chromatic}}]
\label{thm.lexicographic.product}
Let $G$ and $H$ be two graphs. If $\chi(H)=n$, then $\chi(G \cdot H)= \chi(G\cdot K_{n})$.
\end{theorem}

\begin{proposition}
\label{prop.reduction}
    Let $G$ be a graph. Then $\chi(G)\leq k$ if and only if $\mathsf{chor}(G \cdot K_{2}^{c})\leq k$.
\end{proposition}

\begin{proof}[Proof of \autoref{prop.reduction}]
For the forward direction: By \autoref{thm.chordality.and.chromatic.number} and \autoref{thm.lexicographic.product} we have that $\mathsf{chor}(G\cdot K_{2}^{c}) \leq \chi(G\cdot K_{2}^{c}) = \chi(G\cdot K_{1}) = \chi(G) \leq k$.

For the reverse direction: Suppose that $\mathsf{chor}(G\cdot K_{2}^{c})\leq k$ and let $H_{1},\ldots ,H_{k}$ be chordal graphs such that  $G\cdot K_{2}^{c} = H_{1} \cap \ldots \cap H_{k}$. Let $V(K_{2}^{c}) = \{1,2\}$. Let $ f\colon V(G) \rightarrow [k]$ be defined as follows: $f(v) = i\in [k]$, where $i$ is chosen so that it satisfies $\{(v,1), (v,2)\} \notin E(H_{i})$. We claim that $f$ is a proper $k$-coloring of $G$. Suppose not. Let $\{u,v\}\in E(G)$ be such that $f(u)=f(v)=:i$. Then we have that $\{(v,1), (v,2)\} \notin E(H_{i})$ and $\{(u,1), (u,2)\} \notin E(H_{i})$. Thus, $H_{i}[\{(v,1), (v,2), (u,1), (u,2)\}]$ is a hole in $H_{i}$ which is a contradiction. 
\end{proof}

\begin{corollary}
\label{corol.reduction.eq}
Let $G$ be a graph. Then, in polynomial-time in the size of $G$ 
we can construct a graph $G'$ such that the 
following hold: $\chi(G)=k$ if and only if $\mathsf{chor}(G')=k$.
\end{corollary}

Now \autoref{thm.hardness.of.chordality} follows immediately by \autoref{thm.hardeness.of.coloring}
and \autoref{corol.reduction.eq}.

\printbibliography

\end{document}